\documentclass[journal,twoside,web]{ieeecolor}
\usepackage{lcsys}
\usepackage{cite}
\usepackage{amsmath,amssymb,amsfonts}
\usepackage{algorithmic}
\usepackage{graphicx}
\usepackage{textcomp}

\usepackage{stmaryrd} 

\usepackage[hang,small,bf]{caption}
\usepackage[subrefformat=parens]{subcaption}
\usepackage{stfloats}
\fnbelowfloat

\newcommand{\vertiii}[1]{{\left\vert\kern-0.25ex\left\vert\kern-0.25ex\left\vert #1 
    \right\vert\kern-0.25ex\right\vert\kern-0.25ex\right\vert}}

\newcommand{\qedtheorem}{\hfill \ensuremath{\Diamond}}

\newtheorem{dfn}{Definition}

\newtheorem{lem}[dfn]{Lemma}
\newtheorem{thm}[dfn]{Theorem}
\newtheorem{cor}[dfn]{Corollary}
\newtheorem{rem}[dfn]{Remark}
\newtheorem{prob}[dfn]{Problem}

\newtheorem{alg}[dfn]{Algorithm}

\def\BibTeX{{\rm B\kern-.05em{\sc i\kern-.025em b}\kern-.08em
    T\kern-.1667em\lower.7ex\hbox{E}\kern-.125emX}}
\begin{document}
\title{Mutual Information Optimal Control of Discrete-Time Linear Systems}
\author{Shoju Enami and Kenji Kashima
\thanks{This work was supported by JSPS KAKENHI Grant Number 21H04875.}
\thanks{The authors are with the Graduate School of Informatics, Kyoto University, Kyoto, Japan
        (e-mail: enami.shoujyu.57r@st.kyoto-u.ac.jp, kk@i.kyoto-u.ac.jp)}
}

\maketitle
\thispagestyle{empty} 

\begin{abstract}
In this paper, we formulate a mutual information optimal control problem (MIOCP) for discrete-time linear systems. 
This problem can be regarded as an extension of a maximum entropy optimal control problem (MEOCP).
Differently from the MEOCP where the prior is fixed to the uniform distribution, the MIOCP optimizes the policy and prior simultaneously.
As analytical results, under the policy and prior classes consisting of Gaussian distributions, we derive the optimal policy and prior of the MIOCP with the prior and policy fixed, respectively.
Using the results, we propose an alternating minimization algorithm for the MIOCP.
Through numerical experiments, we discuss how our proposed algorithm works.
\end{abstract}


\section{Introduction}\label{sec:introduction}

\IEEEPARstart{A} framework that introduces randomness into actions and control inputs through entropy regularization has been studied mainly in the fields of reinforcement learning (RL)\cite{haarnoja2017reinforcement, haarnoja2018soft} and optimal control \cite{ito2023maximum,ito2024maximum}.
Entropy regularization brings various benefits such as promoting exploration in RL \cite{haarnoja2017reinforcement} and enhancing robustness against disturbances \cite{eysenbach2021maximum}, to name a few.
However, on the other hand, stochastic actions and control inputs may risk worsening the net control cost.
These advantages and disadvantages are attributed to the effect of the entropy regularization term, which makes the probability distributions of inputs approach the uniform distribution in the sense of Kullback–Leibler (KL) divergence.

In \cite{grau2018soft, leibfried2020mutual}, mutual information RL was proposed as an extension of maximum entropy RL to balance control performance and the benefits provided by noise induced by the regularization term.
In mutual information RL, not only the policy but also the prior, which is fixed to the uniform distribution in maximum entropy RL, is optimized.
Practical algorithms that alternately optimize the policy and prior have been proposed.
The experimental results in \cite{grau2018soft, leibfried2020mutual} show that mutual information RL yields better learning outcomes than maximum entropy RL in some problem settings.

However, to the best of our knowledge, there is little to no theoretical work on this topic.
Against this background, addressing the problem setting of linear systems with quadratic costs is an important challenge as a first step toward theoretical understanding \cite{recht2019tour, yang2019provably}.
Motivated by this, in this paper, we consider a mutual information optimal control problem (MIOCP) for linear systems with quadratic costs and Gaussian policy and prior classes.
The main contribution of this paper is the derivation of the optimal policy and prior for a fixed prior and policy in closed form, respectively.
In addition, we propose an alternating minimization algorithm for the MIOCP using the main result and discuss how the algorithm works through numerical examples.
The results of this paper are expected to play an important role as a theoretical foundation for mutual information regularization in the future.
In fact, in maximum entropy optimal control, analysis results in such simple problem settings have elucidated the equivalence with the Schr\"{o}dinger bridge \cite{ito2023maximum} and have been used in Differential Dynamic Programming \cite{so2022maximum}.
These research directions are also important in mutual information optimal control for providing theoretical insights and guarantees, and for advancing it into a practical method; the results of this paper are essential in that regard.

The MIOCP can be formulated as an optimal control problem with a regularization term of the mutual information between the state and the input.
Similarly, \cite{tanaka2017lqg} tackled a control problem of minimizing the directed information from the state to the input under a control performance constraint.
There is a deep relationship between mutual information and directed information: the mutual information between two random variables $x$ and $u$ is expressed as the sum of the directed information from $x$ to $u$ and from $u$ to $x$ (conservation of information \cite{massey2005conservation}).
Reference \cite{tanaka2017lqg} incorporated the directed information under the background of applications such as control under limited state information due to a noisy channel from a plant to a controller, whereas mutual information regularization aims to balance control performance and the benefits provided by noise.

This paper is organized as follows: In Section \ref{sec:Problem Formulation}, we formulate an MIOCP for discrete-time linear systems.
In Section \ref{sec:Theoretical Results}, we derive the optimal policy and prior of the MIOCP with the prior and policy fixed, respectively.
Section \ref{sec:Alternating Minimization Algorithm and Numerical Examples} proposes an alternating minimization algorithm that optimizes the policy and prior alternatively, and then shows some numerical examples of the proposed algorithm.
In Section \ref{sec:Conclusion}, we conclude this paper.

\paragraph*{Notation}
The set of all integers that are larger than or equal to $a$ is denoted by $\mathbb{Z}_{\geq a}$.
The set of integers $\{k,k+1,\ldots, l\}(k<l)$ is denoted by $\llbracket k, l\rrbracket$.
For two symmetric matrices $A, B$ of the same size, we write $A \succ B$ (resp. $A \succeq B$) if $A-B$ is positive definite (resp. positive semi-definite).
The identity matrix is denoted by $I$, and its dimension depends on the context.
The Euclidean norm is denoted by $\|\cdot \|$.
The determinant and trace of $A \in \mathbb{R}^{n\times n}$ is denoted by $|A|$ and $\mathrm{Tr}[A]$.
For $x \in \mathbb{R}^{n}$ and $A \succeq 0$ of size $n$, denote $\|x\|_{A} := (x^{\top}Ax)^{\frac{1}{2}}$.
The expected value of a random variable is denoted by $\mathbb{E}[\cdot]$.
A multivariate Gaussian distribution of dimension $n$ with mean $\mu \in \mathbb{R}^{n}$ and covariance matrix $\Sigma \succ 0$ is denoted by $\mathcal{N}(\mu,\Sigma)$.
When we emphasize that $w$ follows $\mathcal{N}(\mu, \Sigma)$, $w$ is described explicitly as $\mathcal{N}(w | \mu, \Sigma)$.
The KL divergence between probability distributions $p$ and $q$ is denoted by $\mathcal{D}_{\text{KL}}[p \| q]$ when it is defined.
We use the same symbol for a random variable and its realization.
We abuse the notation $p$ as the distribution of a random variable depending on the context.

\section{Problem Formulation} \label{sec:Problem Formulation}
This paper considers the following problem.
\begin{prob}
    Find a pair of a policy $\pi = \{\pi_{k}\}_{k=0}^{T-1}$ and a prior $\rho = \{\rho_{k}\}_{k=0}^{T-1}$ that solves
    \begin{align}
        &\min_{(\pi, \rho)} J(\pi,\rho)\nonumber\\
        &\hspace{20pt}:=\mathbb{E}\left[ \sum_{k=0}^{T-1} \left\{\frac{1}{2}\|u_{k}\|_{R_{k}}^{2} + \varepsilon \mathcal{D}_{\text{KL}}[\pi_{k}(\cdot|x_{k}) \| \rho_{k}(\cdot)] \right\}\right.\nonumber \\
        &\left. \hspace{33pt}+ \frac{1}{2}\|x_{T}-\mu_{x_{\text{fin}}}\|_{F}^{2} \right] \label{eq:objective function of MIOCP without terminal constraint}\\
        &\mbox{s.t. }x_{k+1} = A_{k} x_{k} + B_{k} u_{k} +  w_{k}, \label{eq:linear system}\\
        &\hspace{19pt}u_{k} \sim \pi_{k}(\cdot |x) \ \mbox{given }x=x_{k}, \label{eq:stochastic feedback input}\\
        &\hspace{19pt}w_{k} \sim \mathcal{N}(0, \Sigma_{w_{k}}), \label{eq:process noise}\\
        &\hspace{19pt}x_{0} \sim \mathcal{N}(\mu_{x_{\text{ini}}}, \Sigma_{x_{\text{ini}}}), \label{eq:initial condition}
    \end{align}
    where $\varepsilon >0,  T \in \mathbb{Z}_{\geq 1}, x_{k},\mu_{x_{\text{ini}}}, \mu_{x_{\text{fin}}} \in \mathbb{R}^{n}, u_{k} \in \mathbb{R}^{m}, A_{k} \in \mathbb{R}^{n \times n}, B \in \mathbb{R}^{n \times m}, R_{k}, F, \Sigma_{w_{k}}, \Sigma_{x_{\text{ini}}}\succ 0$.
    A stochastic policy $\pi_{k}(\cdot|x)$ denotes a conditional probability distribution of $u_{k}$ given $x_{k} = x$ and $\rho_{k}(\cdot)$ is a probability distribution on $\mathbb{R}^{m}$. \qedtheorem
    \label{prob:MIOCP without terminal constraint}
\end{prob}

\begin{rem}
    As derived in Remark \ref{rem:derivation of MI from KL}, by optimizing only $\rho$ for Problem \ref{prob:MIOCP without terminal constraint} with $\pi$ fixed, Problem \ref{prob:MIOCP without terminal constraint} can be rewritten as follows:
    \begin{align*}
        &\min_{\pi} \mathbb{E}\left[ \sum_{k=0}^{T-1} \left\{\frac{1}{2}\|u_{k}\|_{R_{k}}^{2} + \varepsilon \mathcal{I}(x_{k},u_{k}) \right\}+ \frac{1}{2}\|x_{T}-\mu_{x_{\text{fin}}}\|_{F}^{2} \right] \nonumber\\
        &\mbox{s.t. }\eqref{eq:linear system}\text{--}\eqref{eq:initial condition},
    \end{align*}
    where
    \begin{align*}
        &\mathcal{I}(x,u):= \int_{\mathbb{R}^{n}}\int_{\mathbb{R}^{m}}p(x,u)\log \frac{p(x,u)}{p(x)p(u)}dudx
    \end{align*}
    is the mutual information between $x$ and $u$.
    This is the reason why we call Problem \ref{prob:MIOCP without terminal constraint} an MIOCP with reference to mutual information RL \cite{grau2018soft,leibfried2020mutual}.\qedtheorem
    \label{rem:origin of mutual information optimal density control}
\end{rem}

\section{Theoretical Results} \label{sec:Theoretical Results}

In this section, we provide analytical results as preparation for the algorithm of Problem \ref{prob:MIOCP without terminal constraint}, which is proposed in Section \ref{subsec:Alternating Minimization Algorithm}.
Because dealing with Problem \ref{prob:MIOCP without terminal constraint} for general policies and priors is challenging, we reduce the difficulty by focusing on Gaussian distributions.
Let us consider the following policy and prior classes.
\begin{align*}
    \mathcal{P} := \{&  \pi = \{ \pi_{k} \}_{k=0}^{T-1} \mid \pi_{k}(\cdot|x) = \mathcal{N}(P_{k}x + q_{k}, \Sigma_{\pi_{k}}),\\
    &P_{k} \in \mathbb{R}^{m\times n}, q_{k} \in \mathbb{R}^{m}, \Sigma_{\pi_{k}} \succ 0 \},\\
    \mathcal{R} := \{&  \rho = \{ \rho_{k} \}_{k=0}^{T-1} \mid \nonumber  \\
    &\rho_{k}(\cdot) = \mathcal{N}(\mu_{\rho_{k}}, \Sigma_{\rho_{k}}), \mu_{\rho_{k}} \in \mathbb{R}^{m}, \Sigma_{\rho_{k}} \succ 0 \}.
\end{align*}
At first glance, the policy class $\mathcal{P}$ seems to restrict the feasible region of $\pi$ overly due to the simplicity of the form of $\pi_{k}$, where the mean is just an affine transformation of $x$ and the covariance matrix does not depend on $x$.
To show the validity of the choice of $\mathcal{P}$, we first show that the policy class $\mathcal{P}$ contains the unique optimal policy of Problem \ref{prob:MIOCP without terminal constraint} with $\rho \in \mathcal{R}$ fixed.
Next, we derive the optimal prior of Problem \ref{prob:MIOCP without terminal constraint} with $\pi \in \mathcal{P}$ fixed.

\subsection{The Optimal Policy for a Fixed Prior}\label{subsec:The Optimal Policy for a Fixed Prior}

We start by introducing the following lemma.
\begin{lem} 
    For a given $\rho \in \mathcal{R},\rho_{k}=\mathcal{N}(\mu_{\rho_{k}},\Sigma_{\rho_{k}})$, define $\Pi_{k}$ as the unique solution to the following Riccati equation:
    \begin{align}
        \Pi_{k} = & A_{k}^{\top} \Pi_{k+1} A_{k} -A_{k}^{\top} \Pi_{k+1} B_{k}\nonumber \\
        &\times \{R_{k} + B_{k}^{\top} \Pi_{k+1} B_{k} + \varepsilon\Sigma_{\rho_{k}}^{-1}\}^{-1}B_{k}^{\top} \Pi_{k+1} A_{k},\nonumber \\
        &k \in \llbracket 0, T-1\rrbracket ,\label{eq:Riccati difference equation}\\
        \Pi_{T} = & F. \label{eq:terminal condition of Riccati difference equation}
    \end{align}
    Then, $\Pi_{k}\succeq 0, k \in \llbracket 0,T\rrbracket$.
    In addition, if $A_{k},k \in \llbracket 0,T-1 \rrbracket$ is invertible, then $\Pi_{k}, k \in \llbracket 0,T\rrbracket$ is also invertible.
    \qedtheorem
    \label{lem:positive semidefiniteness of Pi}
\end{lem}
\begin{proof}
    $\Pi_{T}=F\succ 0 $ holds trivially.
    From \eqref{eq:Riccati difference equation} and the Woodbury matrix identity, we have
    \begin{align}
        &\Pi_{T-1}\nonumber\\
        = &A_{T-1}^{\top}\Pi_{T}^{1/2}\left\{ I+\Pi_{T}^{1/2}B_{T-1}(\varepsilon\Sigma_{\rho_{T-1}}^{-1}+R_{T-1})^{-1} \right.\nonumber\\
        &\left.\times B_{T-1}^{\top}\Pi_{T}^{1/2}\right\}^{-1}\Pi_{T}^{1/2}A_{T-1} \succeq 0. \nonumber 
    \end{align}
    In addition, if $A_{T-1}$ is invertible, $\Pi_{T-1}$ is also invertible.
    By applying this procedure recursively, we obtain the desired result.
\end{proof}

Now, we derive the optimal policy for a fixed prior as follows.

\begin{thm}
    Consider a given prior $\rho \in \mathcal{R}, \rho_{k}(\cdot) = \mathcal{N}(\mu_{\rho_{k}},\Sigma_{\rho_{k}})$.
    Assume that $A_{k}$ is invertible for all $k \in \llbracket 0 ,T-1 \rrbracket$.
    Then, the unique optimal policy $\pi^{\rho}$ of Problem \ref{prob:MIOCP without terminal constraint} with the prior fixed to the given $\rho \in \mathcal{R}$ is given by
    \begin{align}
        \pi_{k}^{\rho}(u | x) \propto \rho_{k}(u) \mathcal{N}(u | \mu_{Q_{k}}, \Sigma_{Q_{k}}), k \in \llbracket 0, T-1 \rrbracket, \label{eq:optimal policy for fixed prior}
    \end{align}
    where
    \begin{align}
        r_{k} = & A_{k}^{-1} r_{k+1} - \varepsilon \Pi_{k}^{-1} A_{k}^{\top} \Pi_{k+1} B_{k} \nonumber \\
        & \times \left\{\Sigma_{\rho_{k}}(R_{k} + B_{k}^{\top} \Pi_{k+1} B_{k}) + \varepsilon I\right\}^{-1}\mu_{\rho_{k}},\label{eq:residual in mean of Q}\\
        &k \in \llbracket 0, T-1 \rrbracket,\nonumber \\
        r_{T} = & \mu_{x_{\text{fin}}}, \label{eq:residual for k=T in mean of Q} \\
        \mu_{Q_{k}} := & -(R_{k} + B_{k}^{\top} \Pi_{k+1} B_{k})^{-1} B_{k}^{\top} \Pi_{k+1} A_{k}\nonumber \\
        &\times(x- A_{k}^{-1}r_{k+1}),\label{eq:mean of Q}\\
        \Sigma_{Q_{k}} := & \varepsilon (R_{k} + B_{k}^{\top} \Pi_{k+1} B_{k})^{-1}. \label{eq:covariance matrix of Q}
    \end{align}
    \qedtheorem
    \label{thm:optimal policy for fixed prior}
\end{thm}

See Appendix I for the proof.
Note that the assumption of Theorem \ref{thm:optimal policy for fixed prior} is not restrictive.
In many situations, $A_{k}$ would be invertible.
For example, if \eqref{eq:linear system} is derived by a zero-order hold discretization of a continuous-time system, $A_{k}$ is always invertible.
From Theorem \ref{thm:optimal policy for fixed prior}, we have the following corollary.
\begin{cor}
    Consider a given prior $\rho \in \mathcal{R}, \rho_{k}(\cdot) = \mathcal{N}(\mu_{\rho_{k}},\Sigma_{\rho_{k}})$.
    Suppose that the same assumptions as Theorem \ref{thm:optimal policy for fixed prior} hold.
    Then, the policy $\pi^{\rho}$ given by \eqref{eq:optimal policy for fixed prior} satisfies that $\pi^{\rho} \in \mathcal{P}$. 
    
    \qedtheorem
    \label{cor:optimal policy for fixed prior is in P}
    
\end{cor}

\begin{proof}
    From a straightforward calculation, a distribution that is proportional to the product of two Gaussian distributions $\mathcal{N}(\mu_{1}, \Sigma_{1})$ and $\mathcal{N}(\mu_{2}, \Sigma_{2})$, where $\mu_{1}, \mu_{2} \in \mathbb{R}^{m}$ and $ \Sigma_{1} , \Sigma_{2} \succ 0$, is given by
    \begin{align*}
        \mathcal{N}(&\Sigma_{2}(\Sigma_{1} + \Sigma_{2})^{-1}\mu_{1} + \Sigma_{1}(\Sigma_{1} + \Sigma_{2})^{-1}\mu_{2},\\
        &\Sigma_{1}(\Sigma_{1} + \Sigma_{2})^{-1}\Sigma_{2}).
    \end{align*}
    By using this formula, the policy \eqref{eq:optimal policy for fixed prior} satisfies 
    \begin{align*}
        \pi_{k}^{\rho}(\cdot | x) = \mathcal{N}(\mu_{\pi_{k}^{\rho}}, \Sigma_{\pi_{k}^{\rho}}),
    \end{align*}
    where
    \begin{align}
        \mu_{\pi_{k}^{\rho}} =& \Sigma_{Q_{k}}(\Sigma_{\rho_{k}} + \Sigma_{Q_{k}})^{-1}\mu_{\rho_{k}} + \Sigma_{\rho_{k}}(\Sigma_{\rho_{k}} + \Sigma_{Q_{k}})^{-1}\mu_{Q_{k}}, \label{eq:mean of optimal policy for fixed prior}\\
        \Sigma_{\pi_{k}^{\rho}} =& \Sigma_{\rho_{k}} (\Sigma_{\rho_{k}} + \Sigma_{Q_{k}})^{-1}\Sigma_{Q_{k}}. \label{eq:covariance matrix of optimal policy for fixed prior}
    \end{align}
    By substituting $\eqref{eq:mean of Q}$ into \eqref{eq:mean of optimal policy for fixed prior}, $\mu_{\pi_{k}^{\rho}}$ turns out to be an affine transformation of $x$.
    In addition, from \eqref{eq:covariance matrix of optimal policy for fixed prior}, $\Sigma_{\pi_{k}^{\rho}}$ does not depend on $x$ because both $\Sigma_{\rho_{k}}$ and $\Sigma_{Q_{k}}$ do not depend on $x$.
    Therefore, $\pi^{\rho} \in \mathcal{P}$ holds for $\rho \in \mathcal{R}$, which completes the proof.
\end{proof}

Theorem \ref{thm:optimal policy for fixed prior} and Corollary \ref{cor:optimal policy for fixed prior is in P} are the reasons why we choose the policy class as $\mathcal{P}$ for the prior class $\mathcal{R}$.

\begin{rem}
    Problem \ref{prob:MIOCP without terminal constraint} with $\rho \in \mathcal{R}$ fixed can be regarded as an extension of a maximum entropy optimal control problem (MEOCP) \cite[Section II]{ito2023maximum}.
    Let us substitute the uniform distribution $p^{\text{uni}}(\cdot)$ on $\mathbb{R}^{m}$ for $\rho_{k}$, which is called an improper prior \cite{kass1996selection} and formally given by $p^{\text{uni}}(u) \propto 1$.
    Then, the KL divergence term in \eqref{eq:objective function of MIOCP without terminal constraint} can be formally rewritten as
    \begin{align*}
        \mathcal{D}_{\text{KL}}[\pi_{k}(\cdot|x)\|p^{\text{uni}}(\cdot)]=-\mathcal{H}(\pi_{k}(\cdot|x)) + (\text{constant}),
    \end{align*}
    where $\mathcal{H}(\pi_{k}(\cdot|x))=-\int_{\mathbb{R}^{m}}\pi_{k}(u|x)\log \pi_{k}(u|x)du$ is the entropy of $\pi_{k}(\cdot|x)$.
    On the basis of this observation, Problem \ref{prob:MIOCP without terminal constraint} with $\rho_{k}(\cdot)$ fixed to the uniform distribution can be regarded as an MEOCP.
    Here, let us show that Theorem \ref{thm:optimal policy for fixed prior} and Corollary \ref{cor:optimal policy for fixed prior is in P} are consistent with the result of \cite{ito2023maximum}.
    To this end, we consider the limit where $\rho_{k}(\cdot)$ converges to the uniform distribution, that is, $\Sigma_{\rho_{k}}=\lambda I$ as $\lambda \rightarrow \infty$.
    As $\lambda \rightarrow \infty$, the solution $\Pi_{k}$ of the Riccati equation \eqref{eq:Riccati difference equation} and \eqref{eq:terminal condition of Riccati difference equation} converges to that of the following Riccati equation.
    \begin{align}
        \hat{\Pi}_{k} = & A_{k}^{\top} \hat{\Pi}_{k+1} A_{k} -A_{k}^{\top} \hat{\Pi}_{k+1} B_{k} \nonumber \\
        &\times (R_{k} + B_{k}^{\top} \hat{\Pi}_{k+1} B_{k})^{-1} B_{k}^{\top} \hat{\Pi}_{k+1} A_{k},k \in \llbracket 0, T-1\rrbracket ,\nonumber\\
        \hat{\Pi}_{T} = & F.\nonumber
    \end{align}
    In addition, the solution $r_{k}$ to \eqref{eq:residual in mean of Q} and \eqref{eq:residual for k=T in mean of Q} converges to
    \begin{align}
        \hat{r}_{k} =& A_{k}^{-1}\hat{r}_{k+1},\hat{r}_{T} = \mu_{x_{\text{fin}}}. \nonumber
    \end{align}
    Furthermore, from \eqref{eq:mean of optimal policy for fixed prior} and \eqref{eq:covariance matrix of optimal policy for fixed prior}, it follows that
    \begin{align*}
       \lim_{\lambda \rightarrow \infty}\mu_{\pi_{k}^{\rho}}= \hat{\mu}_{Q_{k}},\lim_{\lambda\rightarrow \infty}\Sigma_{\pi_{k}^{\rho}}= \hat{\Sigma}_{Q_{k}},
    \end{align*}
    where $\hat{\mu}_{Q_{k}}$ and $\hat{\Sigma}_{Q_{k}}$ are given by substituting $\Pi_{k} = \hat{\Pi}_{k}$ and $r_{k}=\hat{r}_{k}$ into \eqref{eq:mean of Q} and \eqref{eq:covariance matrix of Q}, respectively.
    This observation coincides with \cite[Proposition 1]{ito2023maximum}, where the unique optimal solution to the MEOCP associated with Problem \ref{prob:MIOCP without terminal constraint} is given by $\mathcal{N}(\hat{\mu}_{Q_{k}},\hat{\Sigma}_{Q_{k}})$.
    \qedtheorem
    \label{rem:aspect of generalization of maximum entropy optimal density control}
\end{rem}

\subsection{The Optimal Prior for a Fixed Policy}\label{subsec:The Optimal Prior for a Fixed Policy}

Let us denote the mean and covariance matrix of the state $x_{k}$ by $\mu_{x_{k}}$ and $\Sigma_{x_{k}}$, respectively.
From \eqref{eq:linear system}--\eqref{eq:initial condition}, $\mu_{x_{k}}$ and $\Sigma_{x_{k}}$ evolve as follows under $\pi\in \mathcal{P}, \pi_{k}(\cdot|x)=\mathcal{N}(P_{k}x + q_{k},\Sigma_{\pi_{k}})$.
\begin{align}
    &\mu_{x_{k+1}} = (A_{k} + B_{k}P_{k})\mu_{x_{k}} + B_{k}q_{k}, k \in \llbracket 0, T-1 \rrbracket ,\nonumber \\
    &\mu_{x_{0}} = \mu_{x_{\text{ini}}}, \nonumber \\
    &\Sigma_{x_{k+1}} = (A_{k} + B_{k}P_{k})\Sigma_{x_{k}} (A_{k} + B_{k}P_{k})^{\top} + B_{k} \Sigma_{\pi_{k}} B_{k}^{\top}\nonumber\\
    &\hspace{35pt} +\Sigma_{w_{k}},k \in \llbracket 0, T-1 \rrbracket, \label{eq:evolution of covariance matrix of state}\\
    &\Sigma_{x_{0}} = \Sigma_{x_{\text{ini}}}.\nonumber
\end{align}
Then, the optimal prior for a fixed $\pi \in \mathcal{P}$ is given by the following theorem.

\begin{thm}
    Consider a given policy $\pi\in \mathcal{P},\pi_{k}(\cdot|x)=\mathcal{N}(P_{k}x+q_{k},\Sigma_{\pi_{k}})$.
    Then, the unique optimal prior $\rho^{\pi}$ of Problem \ref{prob:MIOCP without terminal constraint} with the policy fixed to the given $\pi \in \mathcal{P}$ is given by
    \begin{align}
        &\rho_{k}^{\pi}(u) = \mathcal{N}(P_{k} \mu_{x_{k}} + q_{k}, \Sigma_{\pi_{k}} + P_{k} \Sigma_{x_{k}} P_{k}^{\top}),k \in \llbracket 0, T-1 \rrbracket. \nonumber
    \end{align}
    \qedtheorem
    \label{thm:optimal prior for fixed policy}
\end{thm}

\begin{proof}
        Since $\pi$ is fixed, we have
    \begin{align*}
        &\min_{\rho} \mathbb{E}\left[ \sum_{k=0}^{T-1} \left\{\frac{1}{2}\|u_{k}\|_{R_{k}}^{2} + \varepsilon \mathcal{D}_{\text{KL}}[\pi_{k}(\cdot|x_{k}) \| \rho_{k}(\cdot)] \right\}\right.\nonumber \\
        &\left. \hspace{33pt}+ \frac{1}{2}\|x_{T}-\mu_{x_{\text{fin}}}\|_{F}^{2} \right]\\
        \Leftrightarrow & \min_{\rho_{k}} \mathbb{E}\left[\mathcal{D}_{\text{KL}}[\pi_{k}(\cdot|x_{k}) \| \rho_{k}(\cdot)]\right], k \in \llbracket 0, T-1 \rrbracket.
    \end{align*}
    Let us consider the normalization condition $\int_{\mathbb{R}^{m}} \rho_{k}(u) du = 1$ and introduce the Lagrangian multiplier $\lambda \in \mathbb{R}$.
    Then, the Lagrangian of the above problem is given by
    \begin{align*}
        &\mathbb{E}\left[\mathcal{D}_{\text{KL}}[\pi_{k}(\cdot|x_{k}) \| \rho_{k}(\cdot)]\right] + \lambda \left(\int_{\mathbb{R}^{m}} \rho_{k}(u) du - 1 \right)\\
        =&\int_{\mathbb{R}^{n}} \mathcal{N}(x\mid \mu_{x_{k}}, \Sigma_{x_{k}})\left[\int_{\mathbb{R}^{m}} \left\{ \pi_{k}(u | x) \log{\frac{\pi_{k}(u|x)}{\rho_{k}(u)}}  \right.   \right. \\
        &\left. \left. + \lambda \rho_{k}(u) \right\}du - \lambda \right] dx.
    \end{align*}
    By applying the variational method, the infinitesimal variation of the Lagrangian is given by
    \begin{align*}
        \int_{\mathbb{R}^{m}} \left\{ - \frac{\int_{\mathbb{R}^{n}} \mathcal{N}(x\mid \mu_{x_{k}}, \Sigma_{x_{k}}) \pi_{k}(u|x) dx}{\rho_{k}(u)} + \lambda \right\}\delta \rho_{k}(u)du,
    \end{align*}
    where $\delta \rho_{k}(\cdot)$ is the infinitesimal variation of $\rho_{k}(\cdot)$.
    Therefore, the optimal prior $\rho^{\pi}$ of Problem \ref{prob:MIOCP without terminal constraint} with $\pi \in \mathcal{P}$ fixed satisfies
    \begin{align*}
        \frac{\int_{\mathbb{R}^{n}} \mathcal{N}(x\mid \mu_{x_{k}}, \Sigma_{x_{k}}) \pi_{k}(u|x) dx}{\rho_{k}^{\pi}(u)} = \lambda.
    \end{align*}
    It hence follows that
    \begin{align}
        \rho_{k}^{\pi}(u) =& \int_{\mathbb{R}^{n}} \mathcal{N}(x \mid \mu_{x_{k}}, \Sigma_{x_{k}}) \pi_{k}(u|x)dx \label{eq:optimal prior for fixed policy as marginalized distribution}\\
        =&\mathcal{N}(P_{k}\mu_{x_{k}}+q_{k}, \Sigma_{\pi_{k}} + P_{k} \Sigma_{x_{k}} P_{k}^{\top}),\nonumber
    \end{align}
    which completes the proof.
\end{proof}
Note that Theorem \ref{thm:optimal prior for fixed policy} implies that $\rho^{\pi} \in \mathcal{R}$ if $\pi \in \mathcal{P}$.

\begin{rem}
    Let us derive the mutual information term by substituting $\rho = \rho^{\pi}$ into the KL divergence term in \eqref{eq:objective function of MIOCP without terminal constraint} as mentioned in Remark \ref{rem:origin of mutual information optimal density control}.
    From \eqref{eq:optimal prior for fixed policy as marginalized distribution}, $\rho_{k}^{\pi}(\cdot)$ is the probability distribution of $u_{k}$ under $\pi$.
    Then, we have
    \begin{align*}
        &\mathbb{E}[\mathcal{D}_{\text{KL}}[\pi_{k}(\cdot|x_{k})|\rho_{k}^{\pi}(\cdot)]]\\
        =&\int_{\mathbb{R}^{n}}p(x_{k})\left\{\int_{\mathbb{R}^{m}}\pi_{k}(u_{k}|x_{k})\log\frac{\pi_{k}(u_{k}|x_{k})}{\rho_{k}^{\pi}(u_{k})}du_{k}\right\}dx_{k}\\
        =&\int_{\mathbb{R}^{n}}\int_{\mathbb{R}^{m}}p(x_{k},u_{k})\log\frac{\pi_{k}(u_{k}|x_{k})}{p(u_{k})}du_{k}dx_{k}\\
        =&\int_{\mathbb{R}^{n}}\int_{\mathbb{R}^{m}}p(x_{k},u_{k})\log\frac{p(x_{k},u_{k})}{p(x_{k})p(u_{k})}du_{k}dx_{k}=\mathcal{I}(x_{k},u_{k}).
    \end{align*}
    \qedtheorem
    \label{rem:derivation of MI from KL}
\end{rem}

\section{Alternating Minimization Algorithm and Numerical Examples}\label{sec:Alternating Minimization Algorithm and Numerical Examples}

\begin{figure}[htbp]
    \begin{center}
    \begin{tabular}{c}   
      \begin{minipage}[t]{0.5\hsize}
      \centerline{\includegraphics[width=62mm]{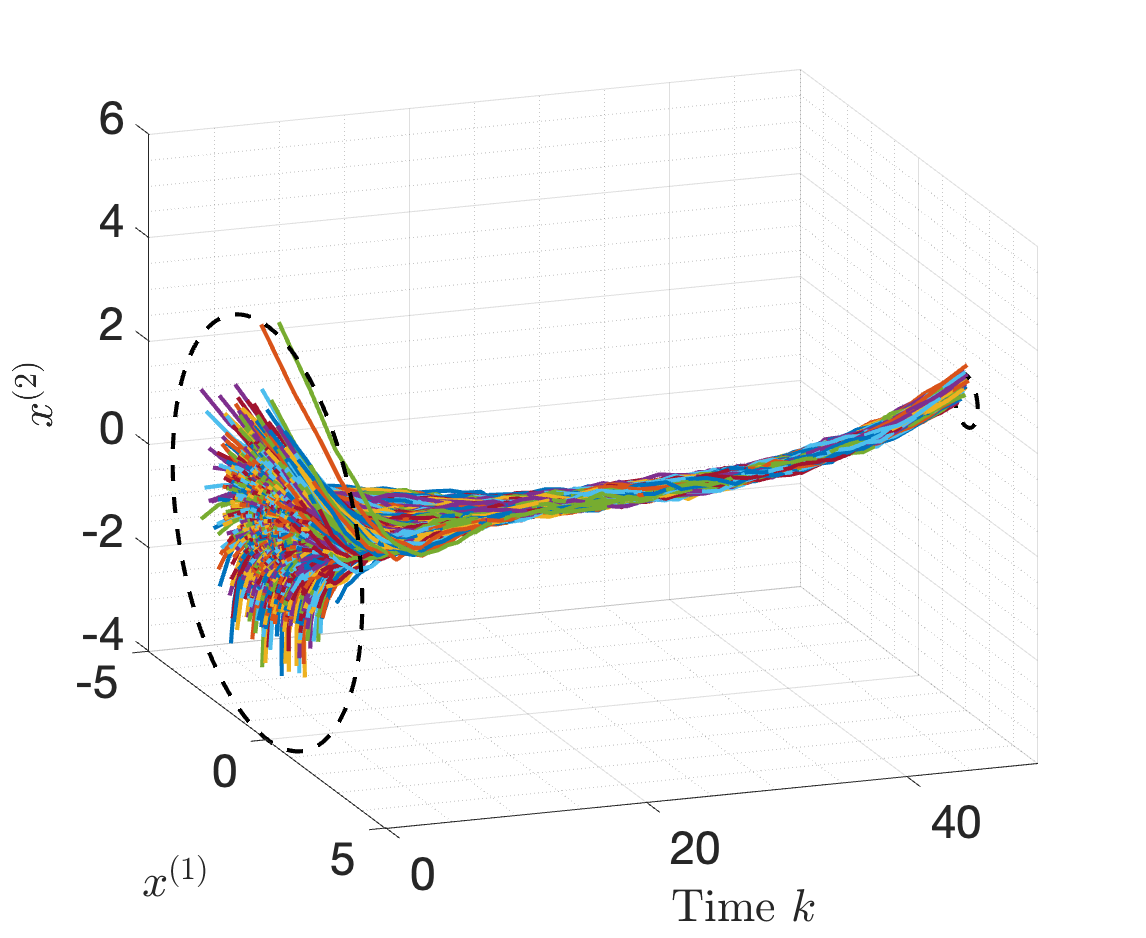}}
        \subcaption{$(i,\varepsilon) =  (10^{5},0.1)$}
        \label{fig:i1000000_eps0p1}
      \end{minipage}\\
      
      \begin{minipage}[t]{0.5\hsize}
        \centerline{\includegraphics[width=62mm]{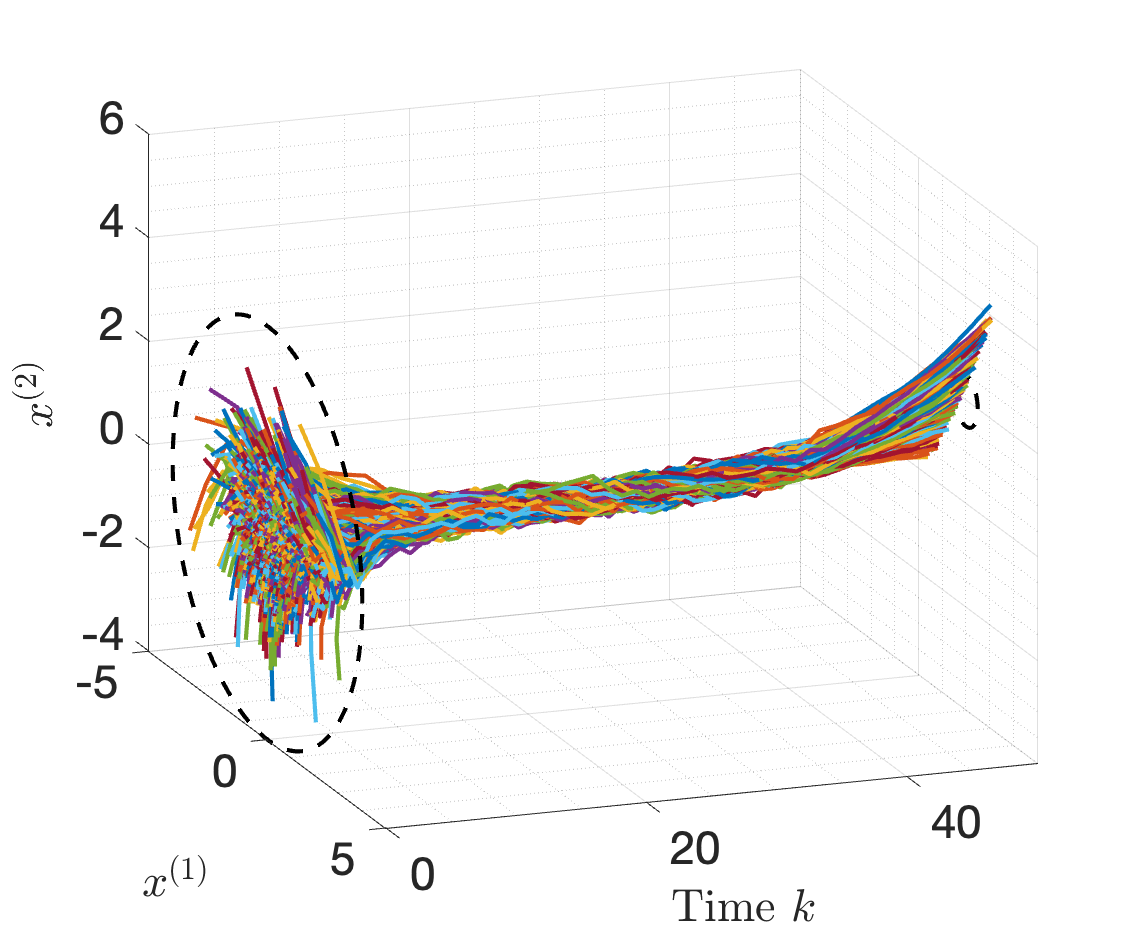}}
        \subcaption{$(i,\varepsilon) = (10^{5},4)$}
        \label{fig:i1000000_eps4}
      \end{minipage}\\

      \begin{minipage}[t]{0.5\hsize}
        \centerline{\includegraphics[width=62mm]{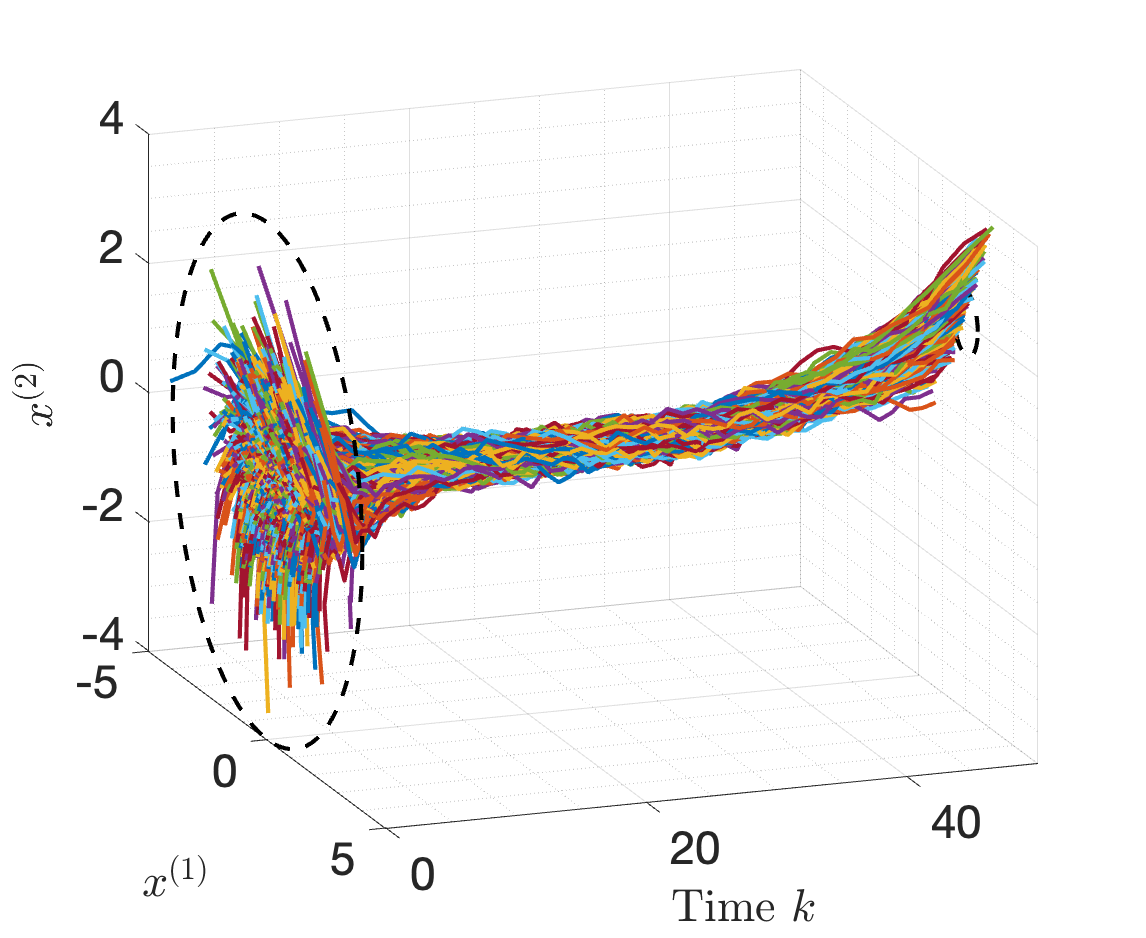}}
        \subcaption{$(i,\varepsilon) =  (20,4)$}
        \label{fig:i20_eps4}
      \end{minipage}
      
    \end{tabular}
    \caption{$1000$ samples of $x_{k}=[x_{k}^{(1)},x_{k}^{(2)}]^{\top}$ (colored lines) under $\pi^{(i)}$.
    The black dashed circles are used to visually highlight the distribution of the initial state $\mathcal{N}(\mu_{x_{\text{ini}}},\Sigma_{x_{\text{ini}}})$ and the terminal target state $\mu_{x_{\text{fin}}}$.}
    \label{fig:sample paths of proposed algorithm}
    \end{center}
\end{figure}

\begin{figure}[htbp]
    \begin{center}
    \centerline{\includegraphics[width=60mm]{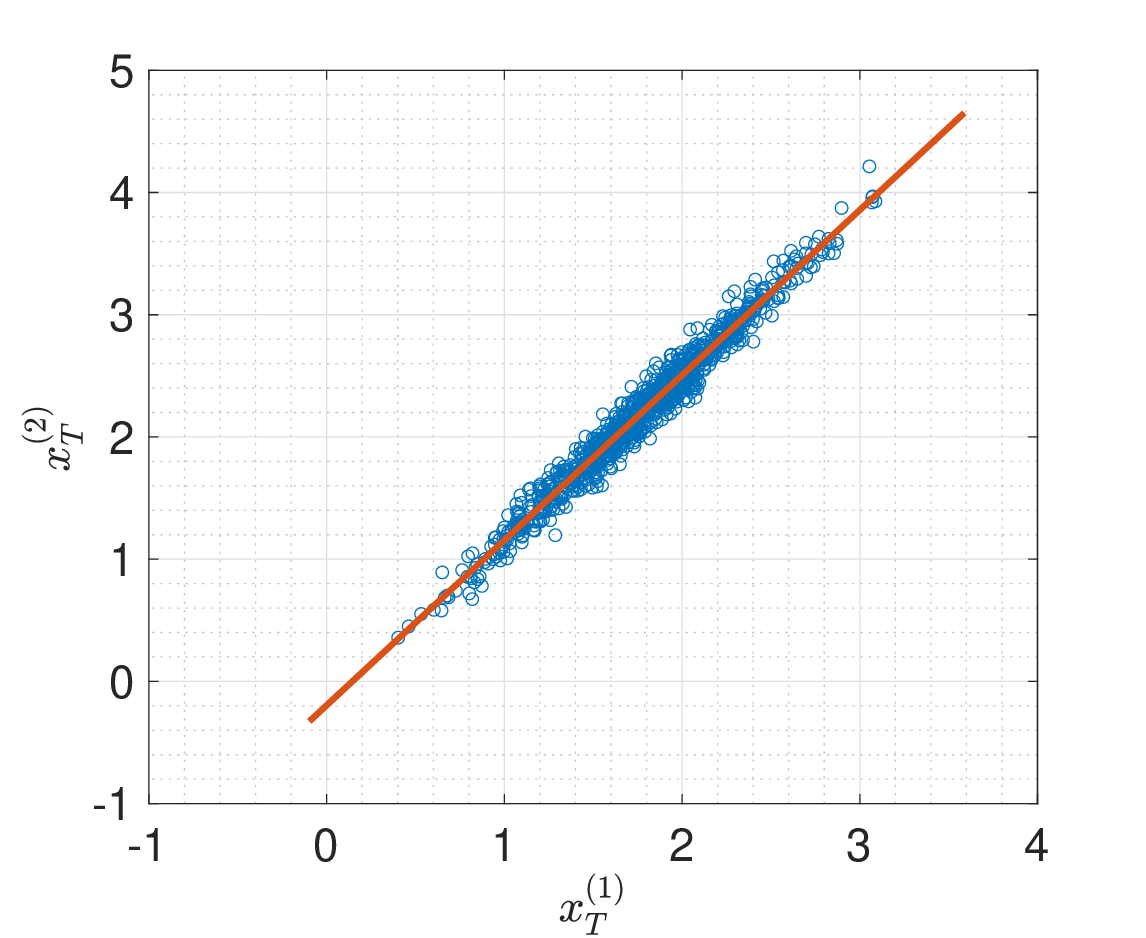}}
    \caption{The $1000$ samples of $x_{T}=[x_{T}^{(1)},x_{T}^{(2)}]^{\top}$ (blue circles) of Fig. \ref{fig:i1000000_eps4} and the linear regression (red line) of the samples given by \eqref{eq:linear regression of the terminal state}.}
    \label{fig:scatter of the terminal state samples}
    \end{center}
\end{figure}

\begin{figure}[htbp]
    \begin{center}
    \centerline{\includegraphics[width=60mm]{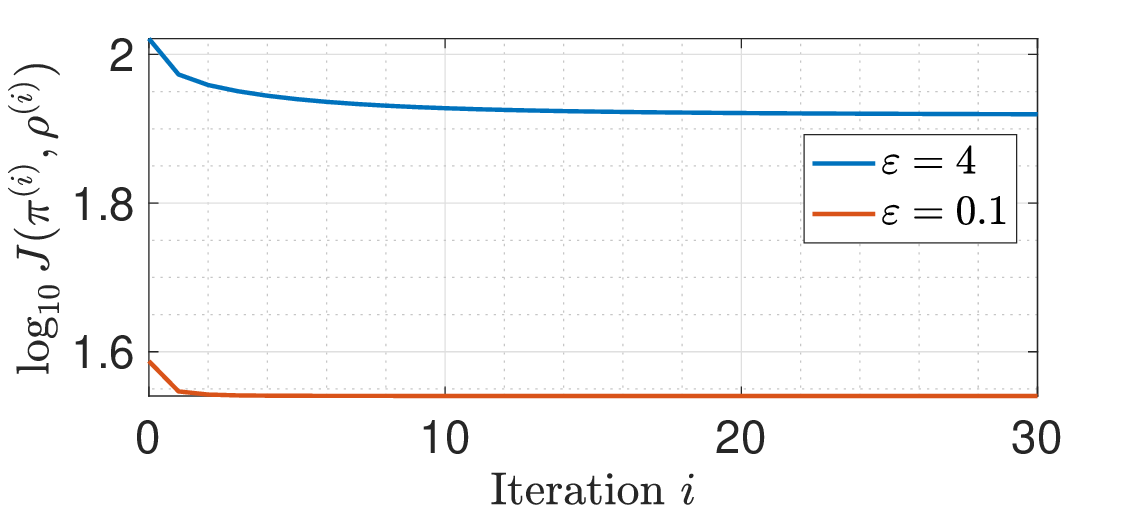}}
    \caption{The trajectories of the objective function $J(\pi^{(i)},\rho^{(i)})$ in semi-log scale.}
    \label{fig:objective function}
    \end{center}
\end{figure}

\begin{figure}[htbp]
    \begin{center}
    \centerline{\includegraphics[width=60mm]{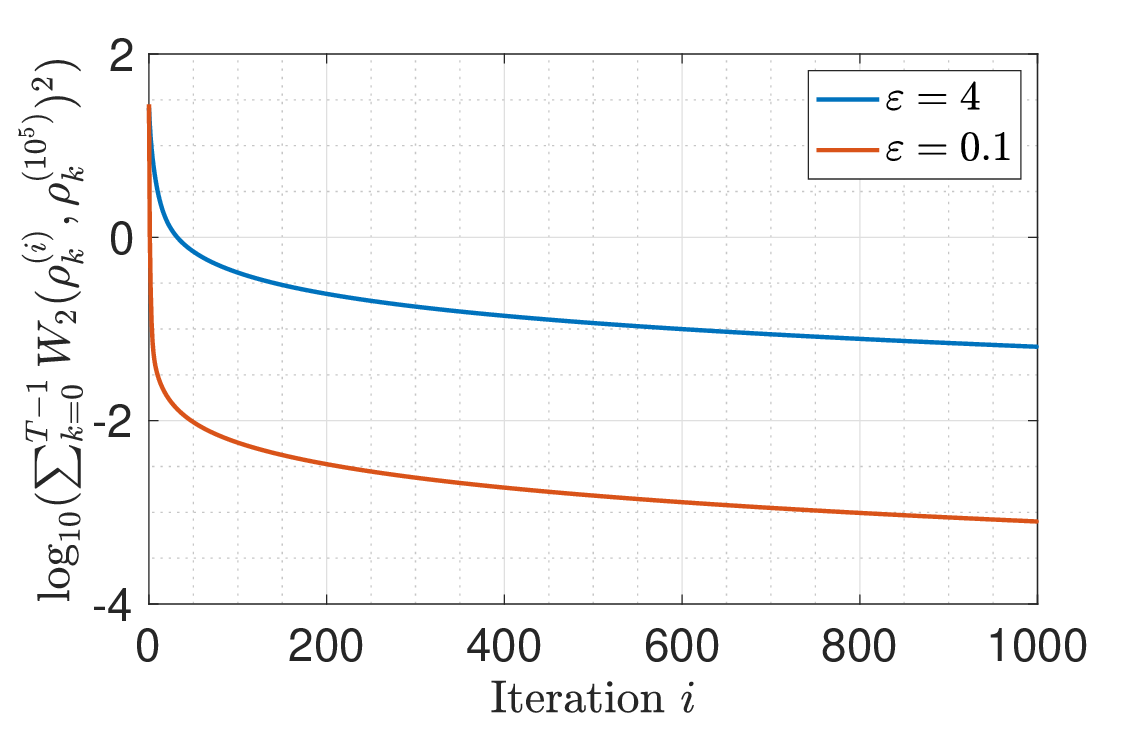}}
    \caption{The trajectories of the sum over time $k$ of the squared Wasserstein $2$-distance between $\rho_{k}^{(i)}$ and $\rho_{k}^{(10^{5})}$ in semi-log
scale.}
    \label{fig:prior distance}
    \end{center}
\end{figure}

In this section, we propose an alternating minimization algorithm on the basis of Section \ref{sec:Theoretical Results} and show some numerical examples of the proposed algorithm.

\subsection{Alternating Minimization Algorithm}\label{subsec:Alternating Minimization Algorithm}
The alternating minimization algorithm is proposed as follows:

\begin{alg}
    \hspace{1pt}
    \begin{description}
        \item[Step 1] Initialize the prior $\rho^{(0)} \in \mathcal{R}$.
        \item[Step 2] Calculate the optimal policy $\pi^{(i)} := \pi^{\rho^{(i)}}$ of Problem \ref{prob:MIOCP without terminal constraint} with $\rho = \rho^{(i)}$ fixed. 
        \item[Step 3] Calculate the optimal prior $\rho^{(i+1)} := \rho^{\pi^{(i)}}$ of Problem \ref{prob:MIOCP without terminal constraint} with $\pi = \pi^{(i)}$ fixed.
        Go back to Step 2. 
    \end{description}
    \qedtheorem
    \label{alg:alternating local minimization algorithm}
\end{alg}

Because $\pi^{\rho} \in \mathcal{P}$ and $\rho \in \mathcal{R}$ for $\rho \in \mathcal{R}$ and $\pi \in \mathcal{P}$, we can calculate $\pi^{(i)} \in \mathcal{P}$ in Step 2 and $\rho^{(i)} \in \mathcal{R}$ in Step 3 for all $i \in \mathbb{Z}_{\geq 0}$ by Theorems \ref{thm:optimal policy for fixed prior} and \ref{thm:optimal prior for fixed policy}, respectively.
Note that Algorithm \ref{alg:alternating local minimization algorithm} supposes that the assumption of Theorem \ref{thm:optimal policy for fixed prior} hold to calculate $\pi^{(i)}$ in Step 2.

\begin{rem}
    For any $\rho^{(0)} \in \mathcal{R}$, the sequence $\{J(\pi^{(i)},\rho^{(i)})\}_{i\in \mathbb{Z}_{\geq 0}}$ generated by Algorithm \ref{alg:alternating local minimization algorithm} is nonincreasing because the policy and the prior are optimized to minimize $J$ at each iteration.
    In addition, $\{J(\pi^{(i)},\rho^{(i)})\}_{i\in \mathbb{Z}_{\geq 0}}$ converges to a nonnegative value because $\{J(\pi^{(i)},\rho^{(i)})\}_{i\in \mathbb{Z}_{\geq 0}}$ is nonincreasing and $J(\pi^{(i)},\rho^{(i)}) \geq 0$ for any $i \in \mathbb{Z}_{\geq 0}$.
    \qedtheorem
    \label{rem:property of proposed algorithm}
\end{rem}

\subsection{Numerical Examples}\label{subsec:Numerical Examples}

This subsection shows numerical examples in the following setting.
The terminal time is given by $T=50$.
The system is given by
\begin{align*}
   &A_{k} = \begin{bmatrix}
        0.9 & 0.2\\
        0.1 & 1.1
    \end{bmatrix},
    B_{k} = \begin{bmatrix}
        0\\
        0.2
    \end{bmatrix}, \Sigma_{w_{k}}=10^{-3}I\ \forall k.
\end{align*}
The eigenvalues of $A_{k}$ are $0.8268$ and $1.1732$ and the corresponding eigenvectors are $[1, -0.3660]^{\top}$ and $[1, 1.3660]^{\top}$, respectively.
The parameters in \eqref{eq:objective function of MIOCP without terminal constraint} are given by
\begin{align*}
    \mu_{x_{\text{fin}}} = [2, 2]^{\top}, F=10I, R_{k} = I, k\in \llbracket 0,T-1 \rrbracket.
\end{align*}
The distribution of the initial state is given by
\begin{align*}
    \mu_{x_{\text{ini}}} = 0, \Sigma_{x_{\text{ini}}}=I.
\end{align*}
The initialized prior $\rho^{(0)},\rho_{k}^{(0)}(\cdot)=\mathcal{N}(\mu_{\rho_{k}^{(0)}}, \Sigma_{\rho_{k}^{(0)}})$ is given by
\begin{align*}
    \mu_{\rho_{k}^{(0)}} = 0, \Sigma_{\rho_{k}^{(0)}}=I\ \forall  k.
\end{align*}

In Fig. \ref{fig:sample paths of proposed algorithm}, the sample paths of the state process under $\pi^{(i)}$ are shown for different $i$ and $\varepsilon$.
The colored lines indicate the samples.
The black dashed circles at $k=0$ in Figs. \ref{fig:i1000000_eps0p1}--\ref{fig:i20_eps4} are the $2\sigma$ covariance ellipses for $\mathcal{N}(\mu_{x_{\text{ini}}},\Sigma_{x_{\text{ini}}})$, which is given by
\begin{align*}
    \{x\in \mathbb{R}^{n}\mid \|x-\mu_{x_{\text{ini}}}\|^{2}_{\Sigma_{x_{\text{ini}}}^{-1}}=2^{2}\}.
\end{align*}
In addition, the black dashed circles at $k=T$ in Figs. \ref{fig:i1000000_eps0p1}--\ref{fig:i20_eps4} highlight the terminal target state $\mu_{x_{\text{fin}}}$, which is given by
\begin{align*}
    \{x\in \mathbb{R}^{n}\mid \|x-\mu_{x_{\text{fin}}}\|=0.5\}.
\end{align*}
Fig. \ref{fig:scatter of the terminal state samples} is the scatter plot of $x_{T}$ of Fig. \ref{fig:i1000000_eps4}.
The blue circles indicate the samples of $x_{T}$ of Fig. \ref{fig:i1000000_eps4} and the red line is the linear regression of the samples, which is given by
\begin{align}
    x_{T}^{(2)} = 1.3508x_{T}^{(1)} - 0.1954.\label{eq:linear regression of the terminal state}
\end{align}
Figs. \ref{fig:objective function} and \ref{fig:prior distance} show semi-log plots of $J(\pi^{(i)},\rho^{(i)})$ and $\sum_{k=0}^{T-1}W_{2}(\rho_{k}^{(i)},\rho_{k}^{(10^{5})})^{2}$ for $\varepsilon = 0.1,4$, respectively, where $W_{2}(p_{1},p_{2})$ is the Wasserstein $2$-distance \cite{villani2008optimal} between two probability distributions $p_{1}$ and $p_{2}$.
For better readability, the trajectories in Figs. \ref{fig:objective function} and \ref{fig:prior distance} are plotted only for $i \in \llbracket0, 30\rrbracket$ and $i\in \llbracket 0,1000 \rrbracket$, respectively.

First, let us discuss how increasing $\varepsilon$ affects the state process.
Because $\varepsilon$ is the coefficient of the KL divergence term, as $\varepsilon$ increases, the policy $\pi$ approaches the prior $\rho$ of the feedforward control law, meaning that $\pi$ will begin to behave like a feedforward control.
It therefore is expected that the first term of \eqref{eq:evolution of covariance matrix of state} can be approximated as $(A_{k}+B_{k}P_{k})\Sigma_{x_{k}}(A_{k}+B_{k}P_{k})^{\top}\simeq A_{k}\Sigma_{x_{k}}A_{k}^{\top}$ when $\varepsilon$ is large.
In addition, $\Sigma_{\pi_{k}}$ tends to be small to decrease the quadratic costs.
Furthermore, a small noise covariance matrix is used in this example.
We therefore approximate \eqref{eq:evolution of covariance matrix of state} roughly as $\Sigma_{x_{k+1}}\simeq A_{k}\Sigma_{x_{k}}A_{k}^{\top}$ when $\varepsilon$ is large.
Then, it is expected that as $\varepsilon$ increases, the state $x_{k}$ will spread into the space spanned by the eigenvectors corresponding to eigenvalues greater than $1$ of $A_{k}$.
In fact, the direction $[1,1.3508]^{\top}$ of the linear regression \eqref{eq:linear regression of the terminal state} is close to the eigenvector $[1,1.3660]$ of $A_{k}$ associated with the eigenvalue greater than $1$.

Next, we consider the convergence criterion for Algorithm \ref{alg:alternating local minimization algorithm}.
In optimization algorithms, convergence criteria are typically determined on the basis of the convergence of the objective function or the decision variables.
As can be seen from Fig. \ref{fig:objective function}, $J(\pi^{(i)},\rho^{(i)})$ decreases monotonically for both $\varepsilon=0.1, 4$, which is consistent with Remark \ref{rem:property of proposed algorithm}.
In addition, it can also be seen that $J(\pi^{(i)},\rho^{(i)})$ has nearly converged after a dozen or so iterations.
On the other hand, from Fig. \ref{fig:prior distance} where $\rho^{(10^{5})}$ is regarded as the tentative convergence destination of $\rho^{(i)}$, the convergence of the decision variables is slower. 
In fact, for $\varepsilon=4$, although the objective function values at $i = 20$ and $i = 10^{5}$ are almost the same, the state processes under $\pi^{(20)}$ exhibit more fluctuations than those under $\pi^{(10^{5})}$, as can be seen in Figs. \ref{fig:i1000000_eps4} and \ref{fig:i20_eps4}.
Given these results, the convergence criterion of Algorithm 9 should be designed on the basis of the convergence of the decision variables.

\section{Conclusion}\label{sec:Conclusion}
This paper considered an MIOCP for linear systems as a min-min optimization problem of the policy and prior.
Under a restriction of the policy and prior classes, we derived the unique optimal policy and prior of the MIOCP with the prior and policy fixed, respectively.
On the basis of this result, we proposed an alternating minimization algorithm for the MIOCP.
Future work includes a more detailed analysis of the optimal solution and the convergence property of the algorithm.
Another direction of future work is elucidating aspects of robust control and the Schr\"{o}dinger bridge in the context of mutual information optimal control, which is expected to clarify the benefits of mutual information optimal control.

\appendices

\section{Proof of Theorem \ref{thm:optimal policy for fixed prior}}
Define the value function associated with Problem \ref{prob:MIOCP without terminal constraint} with $\rho$ fixed as
\begin{align*}
    &V(k,x) := \min_{\pi_{k}}\mathbb{E}\left[ \frac{1}{2}\|u_{k}\|_{R_{k}}^{2} + \varepsilon \mathcal{D}_{\text{KL}}\left[\pi_{k}(\cdot|x)\| \rho_{k}(\cdot)\right] \right. \\
    & \hspace{45pt}\left.+ \mathbb{E}[V(k+1, A_{k}x + B_{k}u_{k}  +w_{k}) ] \mid x_{k} = x\right],\\
    &x \in \mathbb{R}^{n}, k \in \llbracket 0, T-1 \rrbracket,\\
    &V(T,x) := \|x-r_{T}\|_{\Pi_{T}}^{2}, x\in \mathbb{R}^{n}.
\end{align*}
In addition, define the corresponding Q-function as
\begin{align*}
    &Q_{k}(x,u) := \frac{1}{2}\|u\|_{R_{k}}^{2} + \mathbb{E}[ V(k+1, A_{k}x + B_{k}u +w_{k})],\\
    &x\in \mathbb{R}^{n}, u \in \mathbb{R}^{m}, k \in \llbracket 0, T-1 \rrbracket.
\end{align*}
Then, we have
\begin{align*}
    &\mathbb{E}\left[ \frac{1}{2}\|u_{k}\|_{R_{k}}^{2} + \varepsilon \mathcal{D}_{\text{KL}}\left[\pi_{k}(\cdot | x)\| \rho_{k}(\cdot)\right]\right.\\
    &\hspace{10pt}\left.+ \mathbb{E}[V(k+1, A_{k}x + B_{k}u_{k}  +w_{k}) ] \mid x_{k} = x\right]\\
    =&\varepsilon \left\{\mathcal{D}_{\text{KL}}\left[\pi_{k}(\cdot | x) \middle\| \rho_{k}(\cdot)\exp\left(-\frac{1}{\varepsilon }Q_{k}(x, \cdot)\right)/z_{k}\right]\right.\\
    &\hspace{17pt}\left. -\log{z_{k}}\right\},
\end{align*}
where $z_{k}:= \int_{\mathbb{R}^{m}} \rho_{k}(u)\exp\left(-\frac{1}{\varepsilon}Q_{k}(x,u)\right)du$.
Therefore, it follows that
\begin{align}
    \pi_{k}^{\rho}(u|x) \propto \rho_{k}(u)\exp\left(-\frac{1}{\varepsilon}Q_{k}(x, u)\right). \label{eq:optimal policy for fixed prior using Q-function}
\end{align}
Now, let us calculate $Q_{T-1}(x,u)$.
Considering that $\Sigma_{Q_{T-1}}$ is positive definite and $A_{T-1}$ is invertible thanks to Lemma \ref{lem:positive semidefiniteness of Pi} and the assumption of Theorem \ref{thm:optimal policy for fixed prior}, respectively, it follows that
\begin{align*}
    &Q_{T-1}(x,u)\\
    =& \frac{1}{2}\|u\|_{R_{T-1}}^{2} + \frac{1}{2} \mathbb{E}[\|A_{T-1}x + B_{T-1}u-r_{T} +w_{k}\|_{\Pi_{T}}^{2} ]\\
    =& \frac{1}{2}\|u - \mu_{Q_{T-1}}\|_{\varepsilon \Sigma_{Q_{T-1}}^{-1}}^{2} + \frac{1}{2}\|x-A_{T-1}^{-1}r_{T}\|_{\Gamma_{T-1}}^{2} \\
    &+\frac{1}{2}\mathrm{Tr}[\Pi_{T}\Sigma_{w_{T-1}}],
\end{align*}
where $\Gamma_{k} := A_{k}^{\top}\Pi_{k+1}A_{k} - A_{k}^{\top}\Pi_{k+1} B_{k}(R_{k} + B_{k}^{\top}\Pi_{k+1}B_{k})^{-1}B_{k}^{\top}\Pi_{k+1}A_{k}, k\in \llbracket0,T-1 \rrbracket$.
By substituting this into \eqref{eq:optimal policy for fixed prior using Q-function}, we have the policy \eqref{eq:optimal policy for fixed prior} for $k = T-1$.

The value function for $k = T-1$ can be calculated as
\begin{align}
    &V(T-1, x) = -\varepsilon \log{z_{T-1}} \nonumber\\
    = &\frac{1}{2}\|x-A_{T-1}^{-1}r_{T}\|_{\Gamma_{T-1}}^{2}-\frac{\varepsilon}{2}\log\{(2\pi)^{m}|\Sigma_{Q_{T-1}}|\} \nonumber \\
    &+\frac{1}{2}\mathrm{Tr}[\Pi_{T}\Sigma_{w_{T-1}}]\nonumber\\
    & - \varepsilon \log \left\{ \int_{\mathbb{R}^{m}} \mathcal{N}(\mu_{Q_{T-1}},\Sigma_{Q_{T-1}})\rho_{T-1}(u)du\right\}. \label{eq:value function for k = T-1}
\end{align}
The argument of the logarithm of the last term can be calculated as

\begin{align*}
    &\int_{\mathbb{R}^{m}} \mathcal{N}(\mu_{Q_{T-1}},\Sigma_{Q_{T-1}})\rho_{T-1}(u)du\\
    =&\frac{1}{\sqrt{(2\pi)^{m}|\Sigma_{\rho_{T-1}}+\Sigma_{Q_{T-1}}|}}\\
    &\times\exp\left(-\frac{1}{2}\|\mu_{\rho_{T-1}}-\mu_{Q_{T-1}}\|_{(\Sigma_{\rho_{T-1}}+\Sigma_{Q_{T-1}})^{-1}}^{2}\right).
\end{align*}
Therefore, \eqref{eq:value function for k = T-1} can be rewritten as
\begin{align*}
    &V(T-1, x)\\
    = &\frac{1}{2}\|x-A_{T-1}^{-1}r_{T}\|_{\Gamma_{T-1}}^{2}+\frac{\varepsilon}{2}\log\frac{|\Sigma_{\rho_{T-1}}+\Sigma_{Q_{T-1}}|}{|\Sigma_{Q_{T-1}}|}\\
    &+\frac{\varepsilon}{2}\|\mu_{\rho_{T-1}}-\mu_{Q_{T-1}}\|_{(\Sigma_{\rho_{T-1}}+\Sigma_{Q_{T-1}})^{-1}}^{2} +\frac{1}{2}\mathrm{Tr}[\Pi_{T}\Sigma_{w_{T-1}}].
\end{align*}
By substituting \eqref{eq:mean of Q} into $V(T-1,x)$, we have 
\begin{align}
    &V(T-1, x) \nonumber\\
    = &\frac{1}{2}\|x-r_{T-1}\|_{\Pi_{T-1}}^{2} +(\text{Terms independent of }x).\label{eq:arranged value function for k = T-1}
\end{align}
Since the first term of the right-hand side of \eqref{eq:arranged value function for k = T-1} takes the same form as $V(T, x)$, we can derive the policy \eqref{eq:optimal policy for fixed prior} for $k = T-2, T-3, \ldots, 0,$ recursively by following the same procedure as for $k = T-1$, which completes the proof.

\bibliographystyle{IEEEtran}
\bibliography{IEEEabrv, LCSSMIOC.bib}

\end{document}